%% file: main.tex
\documentclass[a4paper, 12pt]{article}

\usepackage[utf8]{inputenc}
\usepackage[T1]{fontenc}
\usepackage[a4paper, margin=2.5cm]{geometry}

\usepackage{libertine} 
\usepackage{inconsolata} 

\usepackage{textpos}

\usepackage{amsmath, amsthm, amssymb}
\usepackage{graphicx}
\usepackage{enumerate}
\usepackage{enumitem}
\usepackage{authblk}
\usepackage[textsize=footnotesize, textwidth=2cm, color=yellow]{todonotes}
\usepackage{thmtools}
\usepackage{thm-restate}
\usepackage[colorlinks=true, citecolor=red]{hyperref}
\usepackage{float}
\usepackage[nameinlink]{cleveref}

\usepackage{mathtools}
\usepackage{xfrac}

\renewenvironment{abstract}
{\small\vspace{-1em}
\begin{center}
\bfseries\abstractname\vspace{-.5em}\vspace{0pt}
\end{center}
\list{}{
\setlength{\leftmargin}{0.6in}%
\setlength{\rightmargin}{\leftmargin}}%
\item\relax}
{\endlist}

\declaretheorem[name=Theorem, numberwithin=section]{theorem}
\declaretheorem[name=Lemma, sibling=theorem]{lemma}
\declaretheorem[name=Proposition, sibling=theorem]{proposition}
\declaretheorem[name=Definition, sibling=theorem]{definition}
\declaretheorem[name=Corollary, sibling=theorem]{corollary}

\declaretheorem[name=Claim, sibling=theorem]{claim}

\def\cqedsymbol{\ifmmode$\lrcorner$\else{\unskip\nobreak\hfil
\penalty50\hskip1em\null\nobreak\hfil$\lrcorner$
\parfillskip=0pt\finalhyphendemerits=0\endgraf}\fi}
\newcommand{\cqed}{\renewcommand{\qed}{\cqedsymbol}}

\def\B{\mathcal{B}}
\def\C{\mathcal{C}}
\def\F{\mathcal{F}}
\def\M{\mathcal{M}}
\def\Ll{\mathcal{L}}
\def\Pp{\mathcal{P}}
\def\X{\mathcal{X}}

\def\S{\mathcal{S}}
\def\T{\mathcal{T}}

\DeclareMathOperator{\adh}{adh}

\DeclareMathOperator{\adhw}{\mathbf{aw}}
\DeclareMathOperator{\bagw}{\mathbf{bw}}

\DeclareMathOperator{\torso}{\mathsf{torso}}

\DeclareMathOperator{\tcc}{3CC}

 \newcommand{\trace}{\mathsf{trace}}

\newcommand{\wh}[1]{\widehat{#1}}

\renewcommand{\leq}{\leqslant}
\renewcommand{\geq}{\geqslant}

\newcommand{\N}{\mathbb{N}}

\newcommand{\quo}[2]{\sfrac{#1}{#2}}

\title{On objects dual to tree-cut decompositions\thanks{ This work is 
a part of projects CUTACOMBS (ŁB, OD, KO) and TOTAL (MP) that have received funding from the European Research Council (ERC) 
under the European Union's Horizon 2020 research and innovation programme (grant agreements No.~714704 and No.~677651, respectively).
}}

\author[1]{\L{}ukasz Bo\.zyk}
\author[1,2]{Oscar Defrain}
\author[1]{\\Karolina Okrasa}
\author[1]{Micha\l{} Pilipczuk}
\affil[1]{Institute of Informatics, University of Warsaw, Poland}
\affil[2]{LIS, Aix-Marseille Universit\'e, France}
\date{26 March, 2021}

\begin{document}

\maketitle

\begin{textblock}{20}(-1.9, 7.7)
\includegraphics[width=40px]{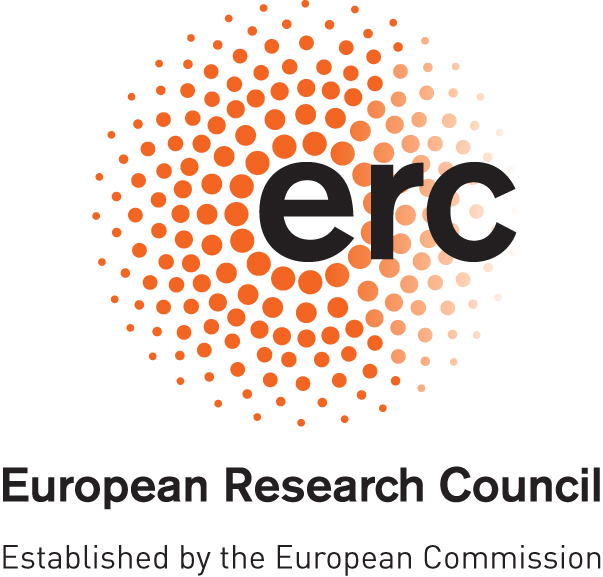}%
\end{textblock}
\begin{textblock}{20}(-2.15, 8.0)
\includegraphics[width=60px]{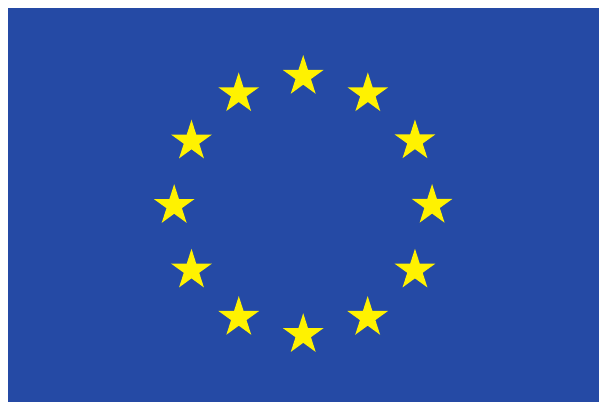}%
\end{textblock}

\input{abstract}

\section{Introduction}\label{sec:intro}
\input{intro}

\section{Preliminaries}\label{sec:prelim}
\input{preliminaries}

\section{Objects}\label{sec:objects}
\input{objects}


\section{Proof of Theorem~\ref{thm:trinity}}\label{sec:proof}
\input{proof}


\bibliographystyle{alpha}
\bibliography{main}

\end{document}

%% file: abstract.tex
\begin{abstract}
Tree-cut width is a graph parameter introduced by Wollan~\cite{wollan2015structure} that is an analogue of treewidth for the immersion order on graphs in the following sense: the tree-cut width of a graph is functionally equivalent to the largest size of a wall that can be found in it as an immersion. In this work we propose a variant of the definition of tree-cut width that is functionally equivalent to the original one, but for which we can state and prove a tight duality theorem relating it to naturally defined dual objects: appropriately defined brambles and tangles. Using this result we also propose a game characterization of tree-cut width.

\noindent{\textbf{Keywords :} immersions, tree-cut decompositions, brambles, tangles, cops and robber game}
\end{abstract}

%% file: intro.tex

A graph $H$ is an {\em{immersion}} of a graph $G$ if one can injectively map vertices of $H$ to vertices of $G$ and edges of $H$ to pairwise edge-disjoint paths in $G$ so that the image of every edge of $H$ connects the images of its endpoints. Immersibility is a natural embedding notion for graphs that, compared to the concept of minors, focuses on edge-disjointness rather than on vertex-disjointness. Similarly to the minor order, the immersion order is a well-quasi-ordering on all graphs~\cite{robertson2010graph}.
This suggests that a sound structural theory can be built around this concept.

And indeed, in~\cite{wollan2015structure} Wollan introduced the graph parameter {\em{tree-cut width}} which can be regarded as the analogue of treewidth for immersions. More precisely, as shown in~\cite{wollan2015structure}, tree-cut width admits a duality theorem of the following form: tree-cut width of a graph $G$ is functionally equivalent to the largest size of a wall that can be immersed in $G$, in the sense that each of these quantities can be bounded by a function of the other. This implies that a class of graphs $\C$ has a uniformly bounded tree-cut width if and only if all the graphs from $\C$ exclude some fixed subcubic planar graph as an immersion. The wall immersion theorem of Wollan is a fundamental structural connection that mirrors the grid minor theorem from the setting of minors and treewidth~\cite{robertson1986graph}.

The definition of tree-cut width is based on the notion of a {\em{tree-cut decomposition}}. A tree-cut decomposition of a graph $G$ is just a tree $T$ where each node $t$ is associated with a {\em{bag}} $X_t\subseteq V(G)$ so that the bags are pairwise disjoint and together cover the whole vertex set. Note that we allow bags to be empty. Thus, every edge $e$ of $T$ naturally corresponds to a partition of $V(G)$ into two parts, respectively consisting of vertices residing in bags of the two connected components of $T-e$. The set of edges of $G$ crossing this partition is called the {\em{adhesion}} of $e$. Roughly speaking, a tree-cut decomposition as above has a bounded width if all of the following quantities are bounded:
\begin{itemize}[nosep]
 \item the size of each bag;
 \item the size of each adhesion; and
 \item the degree of each node in $T$.
\end{itemize}
However, in the last point there is a technical caveat. When calculating the degree of a node of $t$ in $T$, we only consider incident edges that have {\em{bold}} adhesions, that is, adhesions of size $\geq 3$. This detail is necessary to make the connection with immersions of walls valid, but the consequence is that working with tree-cut width is often plagued with technicalities concerning {\em{thin}} adhesions, that is, adhesions of size at most $2$.

We remark that Wollan~\cite{wollan2015structure} used a more convoluted, and arguably more unwieldy definition of the width of a tree-cut decomposition (see Section~\ref{sec:prelim} for a discussion). The understanding presented above was developed by Giannopoulou et al.~\cite{giannopoulou2016linear}, who proposed another, roughly equivalent way of measuring the width. The width measure of Giannopoulou et al. is easier to control during modifications of a tree-cut decomposition, but also has its own share of~problems.

Tree-cut width and ideas around tree-cut decompositions have so far found several applications, mostly connected with immersions. The original motivation of Wollan was to use tree-cut width and the wall immersion theorem to give a structure theorem for graphs excluding a fixed immersion~\cite{wollan2015structure}, a theorem that was obtained independently by \cite{devos2013note}.
These ideas were later used by Dvo\v{r}\'ak and Wollan~\cite{dvovrak2016structure} to give a similar result for the setting of excluding {\em{strong immersions}}; this is a different, but closely related embedding notion. The connections between tree-cut decompositions and wall immersions also turned out to be important in the proof of Liu of the Erd\H{o}s-P\'osa property for immersion models in $4$-edge-connected graphs~\cite{liu2015packing}. On the algorithmic side, Giannopoulou et al.~\cite{giannopoulou2016linear} used the excluded wall immersion theorem of Wollan~\cite{wollan2015structure} and the combinatorics of tree-cut decompositions to give linear kernels for edge deletion problems for immersion-closed classes that exclude at least one subcubic planar graph as an immersion. Dynamic programming on tree-cut decompositions was explored by Ganian et al.~\cite{Ganian0S15}, and an FPT $2$-approximation algorithm for tree-cut width was proposed by Kim et al.~\cite{KimOPST18}. More recently, Giannopoulou et al.~\cite{GiannopoulouKRT19,GiannopoulouKRT21} proved a result about the existence of {\em{lean}} tree-cut decompositions of optimum width, which is an analogue of a classic result of Thomas for tree decompositions~\cite{thomas1990menger}. Using this, they obtained computable upper bounds on the sizes of immersion obstructions for having tree-cut width $\leq k$~\cite{GiannopoulouKRT19}.

\paragraph*{Our contribution.} We believe that despite extensive work on tree-cut width and tree-cut decompositions, the currently used definitions of width---due to Wollan~\cite{wollan2015structure} and to Gianno\-poulou et al.~\cite{giannopoulou2016linear}---still seem somewhat unnatural or arbitrary, and hence not completely understood. The main motivation of this work is to find a definition for tree-cut width that would feel more ``right''.

Comparing the situation with tree decompositions and treewidth, an aspect that convinces us about the naturalness of the definition of treewidth is the existence of naturally defined tightly dual objects: {\em{brambles}} and {\em{tangles}}. Precisely, as proved in part by Robertson and Seymour in \cite{robertson1991graph}, and later by Seymour and Thomas in \cite{seymour1993graph} (see also \cite{diestel2012graph}), the following conditions are equivalent for a graph $G$ and a positive integer $k$:
\begin{itemize}[nosep]
 \item $G$ has no tree decomposition of width $<k$;
 \item $G$ has a bramble of order $\geq k$; and
 \item $G$ has a tangle of order $\geq k$.
\end{itemize}
A corollary of this result is an elegant characterization of treewidth in terms of the cops and robber game~\cite{seymour1993graph}.

The duality theorem for treewidth described above is by now considered a fundamental result in structural graph theory. Several different proofs can be found in the literature~\cite{seymour1993graph,mazoit-brambles,diestel2012graph}. The concepts of brambles and tangles have been also generalized by Diestel to abstract separation systems~\cite{diestel2018abstract}, culminating in a general theorem of Diestel and Oum~\cite{diestel2021tangle}, which states that if a separation system admits a number of technical conditions, then one can derive a suitable tight duality theorem for (appropriately defined) tree decompositions and tangles.
See the article of Diestel and Oum~\cite{diestel2019tangle} for a broader discussion of applications of this theorem to particular settings. Also, we remark that the work of Diestel and Oum was preceded by earlier works~\cite{amini2009submodular,lyaudet2010partitions} on abstract definitions of brambles.

In this work, we apply the same principle to tree-cut decompositions and tree-cut width. That is, we propose a different (but functionally equivalent) measure of the width of a tree-cut decomposition, and present natural analogues of brambles and tangles for which we can state and prove the following result: For a graph $G$ and a positive integer $k$, the following conditions are equivalent:
\begin{itemize}[nosep]
 \item $G$ has no tree-cut decomposition of width $<k$;
 \item $G$ has a bramble of order $\geq k$; and
 \item $G$ has a tangle of order $\geq k$.
\end{itemize}
We remark that in the last point, we mean (appropriately defined) tangles of {\em{edge separations}}: partitions of the vertex set, where the edges crossing the partition form the cutset. Tangles of this flavor were already considered: Liu~\cite{liu2015packing} used them extensively in the setting of the Erd\H{o}s-P\'osa property for immersion models, while Diestel and Oum~\cite{diestel2019tangle} gave a suitable tangle duality theorem for the parameter {\em{carving-width}}, which is related to tree-cut width.

An important aspect that emerged during our work is that tree-cut decompositions naturally have not one, but two orthogonal width measures. The {\em{bag-width}} is defined as the maximum size of a bag, and the {\em{adhesion-width}} (roughly) governs the sizes of adhesions. Mirroring this, brambles and tangles have two measures of order: {\em{bag-order}} and {\em{adhesion-order}}. We can prove the tight duality result also in this biparametric setting: For a graph $G$ and positive integers $a$ and $b$, the following conditions are equivalent:
\begin{itemize}[nosep]
 \item $G$ has no tree-cut decomposition of adhesion-width $<a$ and bag-width $<b$;
 \item $G$ has a bramble of adhesion-order $\geq a$ and bag-order $\geq b$; and
 \item $G$ has a tangle of adhesion-order $\geq a$ and bag-order $\geq b$.
\end{itemize}
In the hindsight, the biparametric aspect of tree-cut decompositions is visible in the previous definitions~\cite{wollan2015structure,giannopoulou2016linear}, but the two parameters were combined into a single width measure in a somewhat arbitrary way. We believe that handling the two parameters separately actually clarifies the situation.
We also note that a biparametric variant of treewidth has been considered in \cite{geelen2016generalization}.

Similarly to the setting of treewidth, we can use our duality theorem to give a characterization of tree-cut width in terms of a search game that we call {\em{cops, dogs, and robber game}}. The game is played similarly to the standard cops and robber game, with the exception that the $a$ cops are now always placed on edges of the graph, instead of vertices. The robber is always placed on a vertex and may move freely along paths that avoid edges occupied by cops\footnote{There is an additional technical rule that mirrors the exceptional treatment of thin adhesions.}. Obviously, cops placed on edges cannot directly catch the robber placed on a vertex, but they also have $b$ dogs for this purpose. Namely, once the size of the set of vertices to which the robber can possibly move is limited to $\leq b$, the cops can unleash the dogs on those vertices and immediately catch the robber. We prove that a graph has a tree-cut decomposition with adhesion-width $\leq a$ and bag-width $\leq b$ if and only if there is a strategy to catch the robber using $a$ cops and $b$ dogs.

As for the proofs, our argumentation closely follows the standard strategy used for the duality theorem for treewidth, but adjusted to the combinatorics of tree-cut decompositions. As we mentioned, this strategy is by now very well understood~\cite{seymour1993graph,mazoit-brambles,diestel2012graph}, even in the general setting of abstract separation systems~\cite{diestel2018abstract,diestel2019tangle}. In fact, we expect that much of the technical work presented in this paper, especially concerning the duality of tree-cut decompositions and tangles, can be derived from the general theorem of Diestel and Oum~\cite{diestel2021tangle}, similarly as has been done for carving-width in~\cite[Section~5]{diestel2019tangle}. However, we believe that following such a path would actually oppose one of our main goals: clarifying the combinatorics of tree-cut width in a transparent and self-contained manner. More precisely, applying the theorem of Diestel and Oum is not automatic, it requires recalling a list of abstract notions and verifying multiple technical properties of the specific separation systems that we work with. This would result in a proof that would be less transparent and not necessarily any shorter. On the other hand, a point that we would like to make by giving a direct proof is that once all the definitions are rightly set, the standard and well-understood proof strategy applies without any problems, providing a reasoning that is conceptually even simpler than that for the treewidth. 


\paragraph*{Organization.} In Section~\ref{sec:prelim} we set up the notation, recall standard notions and facts, and discuss the definitions of tree-cut width of Wollan~\cite{wollan2015structure} and of Giannopoulou et al.~\cite{giannopoulou2016linear}. In Section~\ref{sec:objects} we introduce our notion of width of a tree-cut decomposition and define the dual objects: brambles and tangles. Section~\ref{sec:proof} is devoted to the proof of the main duality theorem. In Section~\ref{sec:game} we derive the game characterization of tree-cut width.

%% file: preliminaries.tex

\paragraph*{Partitions.}
A \emph{near-partition} of a set $\Omega$ is a family $\X$ of pairwise-disjoint subsets of $\Omega$ such that $\Omega=\bigcup \X$. Note that we allow the elements of a near-partition to be empty. If this is not the case, then $\X$ is a {\em{partition}} of $\Omega$.

\paragraph*{Graphs.} All graphs considered in this paper are finite. We allow the existence of {\em{parallel edges}} (multiple edges with same pair of endpoints) but we do not allow loops (edges with both endpoints at the same vertex). All graphs are undirected unless explicitly stated.

Let $G$ be a graph. By $V(G)$ and $E(G)$ we denote the sets of vertices and of edges of $G$, respectively. For a vertex subset $X$, we define $\delta(X)$ to be the set of all edges with one endpoint in $X$ and the other in $V(G)\setminus X$. The \emph{degree} of a vertex $v$ is $|\delta(\{v\})|$.

A {\em{path}} in $G$ is a connected subgraph of $G$ where every vertex has degree $2$ apart from exactly two vertices of degree one, called the {\em{endpoints}} of the path. A {\em{cycle}} in $G$ is a connected subgraph of $G$ where every vertex has degree $2$. The {\em{length}} of a path or a cycle is defined as its edge count. Two vertices connected by a pair of parallel edges are considered a cycle of length $2$.

For a set of edges $F\subseteq E(G)$, we say that a pair of vertices $u,v$ is {\em{disconnected}} by $F$ if $u$ and $v$ belong to different connected components of $G-F$.

An \emph{(edge) separation}\footnote{We remark that most of literature in structural graph theory use term {\em{separation}} for a vertex separation, that is a pair $\{A,B\}$ of subsets of vertices with $A\cup B=V(G)$ and no edge between $A\setminus B$ and $B\setminus A$. Throughout this work we will solely work with edge separations as defined above, hence for brevity we call them simply {\em{separations}}.} in $G$ is a near-partition $\{A,B\}$ of $V(G)$ that has two elements, called the {\em{sides}}.
The \emph{order} of the separation $\{A,B\}$ is $|\delta(A)|=|\delta(B)|$. Separations of order at most $2$ are called \emph{thin}, whereas those of order at least $3$ are called \emph{bold}.

\newcommand{\tccrel}{\sim_{3\textrm{CC}}}

\paragraph*{$3$-edge-connectedness.}
Two vertices $u$ and $v$ in a graph $G$ are {\em{$3$-edge-connected}} if there exist three edge-disjoint paths connecting $u$ and $v$. By Menger's Theorem, the relation of $3$-edge-connectedness is an equivalence relation on $V(G)$. We will denote this relation by $\tccrel$; the graph $G$ will always be clear from the context. The equivalence classes of $\tccrel$ are called the {\em{$3$-edge-connected components}} of $G$.

A set of vertices $X\subseteq V(G)$ is {\em{$3$-edge-connected}} in $G$ if the vertices of $X$ are pairwise $3$-edge-connected in $G$. Equivalently, $X$ is $3$-edge-connected if it is entirely contained in a single $3$-edge-connected component of $G$.

Suppose $\sim$ is an equivalence relation on the vertex set of $G$. We define the {\em{quotient graph}} $\quo{G}{\sim}$ as follows. The vertices of $\quo{G}{\sim}$ are the equivalence classes of $\sim$, and each edge $uv$ of $G$ with $u\not\sim v$ gives rise to one edge $AB$ in $\quo{G}{\sim}$, where $A$ and $B$ are the equivalence classes of $\sim$ to which $u$ and $v$ belong, respectively. Note that thus, the number of parallel edges in $\quo{G}{\sim}$ connecting a pair $A,B$ of equivalence classes of $\sim$ is equal to the number of edges in $G$ whose one endpoint is in $A$ and the other is in $B$. Also, edges of $G$ whose endpoints belong to the same equivalence class of $\sim$ do not contribute to the edge set of $\quo{G}{\sim}$.

In Section~\ref{sec:3cc} we will study the structure of the quotient graph $\quo{G}{\tccrel}$, for any graph~$G$.

\paragraph*{Immersions.} 
We say that a graph $G$ admits a graph $H$ as an \emph{immersion} if there exists an {\em{immersion model}} of $H$ in $G$: a mapping $\pi$ defined on the vertices and edges of $H$ as follows:
\begin{itemize}[nosep]
\item $\pi$ maps vertices of $H$ to pairwise different vertices of $G$;
\item $\pi$ maps each edge $e$ of $H$ with endpoints $u$ and $v$ to a path in $G$ with endpoints $\pi(u)$ and $\pi(v)$; and
\item paths in $\{\pi(e)\colon e\in E(H)\}$ are pairwise edge-disjoint.
\end{itemize}
{\em Walls} are central to the notions presented in this paper.
A {\em $k \times k$ wall} is a graph on $2k^2$ vertices, constructed from $k$ (horizontal) paths $P_1,\dots,P_k$. Each $P_i$ has vertex set $\{v^i_1,\dots,v^i_{2k}\}$ where $v^i_j$ and $v^i_{j+1}$ are adjacent for all $1\leq j\leq 2k-1$, and there are the following additional edges between the paths:
\begin{itemize}[nosep]
	\item $v^i_jv^{i+1}_j$ if $i,j$ are odd, $1\leq i<k$, $1\leq j\leq 2k$;
 	\item $v^i_jv^{i+1}_j$ if $i,j$ are even, $1\leq i<k$, $1\leq j\leq 2k$.
\end{itemize}

\paragraph{Tree-cut width.}
We now recall the concept of tree-cut width, as defined by Wollan~\cite{wollan2015structure}, and then give an equivalent definition proposed by Giannopoulou et al.~\cite{giannopoulou2016linear}.
First, we need a notion of a decomposition.

\begin{definition}[Tree-cut decomposition]\label{def:tc-dec}
	A {\em{tree-cut decomposition}} of a graph $G$ is a pair $\T=(T,\X)$ such that $T$ is a tree and $\X=\{X_t\subseteq V(G) \mid t\in V(T)\}$ is a near-partition of the vertex set of $G$, indexed by the nodes of $T$:
	for every $t\in V(T)$, the set $X_t\in \X$ is called {\em bag} of $t$.
\end{definition}

We now introduce some useful definitions, which will eventually lead to a concept of the width of a tree-cut decomposition. Let us fix a tree-cut decomposition $\T=(T,\X)$ of a graph~$G$. For a pair of vertices $u,v$ of $G$ (not necessarily distinct), the {\em{trace}} of $\{u,v\}$ in $\T$ is the (unique) path in $T$ connecting the node $s$ satisfying $u\in X_s$, and the node $t$ satisfying $v\in X_t$. Note that if $u=v$, then the trace of $\{u,v\}$ consists of only one node, the one whose bag contains $u$. The trace of an edge of $G$ is the trace of its endpoints.

For an edge $st$ of $T$, the {\em{adhesion}} of $st$, denoted $\adh(st)$, is the set of all edges of $G$ whose traces contain $st$. Equivalently, an edge $uv\in E(G)$ belongs to $\adh(st)$ if and only if the node $s'$ satisfying $u\in X_{s'}$ and the node $t'$ satisfying $v\in X_{t'}$ lie in different connected components of $T-st$.
An edge $st$ of $T$ is called {\em{thin}} if $|\adh(st)|\leq 2$, and {\em{bold}} otherwise. In the notation $\adh(\cdot)$, the decomposition $\T$ can be specified in the subscript if it is not clear from the context.

Next, for each node $t$ of $T$ we define the {\em{torso}} $G_t$. Intuitively, $G_t$ is obtained from $G$ by identifying, for every connected component $C$ of $T-t$, all the vertices residing in the bags of $C$ into a single vertex. Formally, $G_t$ is defined as the quotient graph $\quo{G}{\sim_t}$, where $\sim_t$ is an equivalence relation on $V(G)$ defined as follows: $u\sim_t v$ if either $u=v$ or $u\neq v$ and the trace of $\{u,v\}$ in $\T$ does not contain $t$.

Finally, the {\em{$3$-center}} $\widetilde{G}_t$ of the torso $G_t$ is obtained from $G_t$ by iteratively suppressing vertices of degree at most $2$ in $G_t$, but only those that do not belong to $X_t$. Here, to {\em{suppress}} a vertex of degree at most $2$ means to either delete it, provided it had at most one neighbor, or delete it and add an edge connecting its two former neighbors, provided it had exactly two neighbors. It is not hard to see that the order of performing the suppressions does not matter.
However, note that the suppression of a vertex of degree $2$ can reduce the degree of another vertex outside of $X_t$, leading in turn to its suppression.

With all these notions in place, we can recall the definition of tree-cut width originally proposed by Wollan in~\cite{wollan2015structure}.

\begin{definition}[Wollan's tree-cut width]\label{def:tctw-Wollan}
The \emph{Wollan's width} of a tree-cut decomposition $\T=(T,\X)$ of a graph $G$ is defined as \[\max\left(\{|\adh(e)|\colon e\in E(T)\}\cup \{|V(\widetilde{G}_t)|\colon t \in V(T)\}\right).\]
The \emph{Wollan's tree-cut width} of $G$ is the minimum width of a tree-cut decomposition of~$G$. 
\end{definition}

In order to simplify arguments in their study of algorithmic problems related to immersions, Gianno\-poulou et al. introduced in~\cite{giannopoulou2016linear} an alternative definition of tree-cut width.
Let $G$ be a graph and $\T=(T,\X)$ be a tree-cut decomposition of $G$.
They define, for every node $t$ of $T$, the following quantity:
\[
	w_t \coloneqq |X_t| + |\{t'\in N_T(t) \colon |\adh(tt')|\geq 3\}|.
\]
They then propose the following adjustment of the definition of tree-cut width, which they show to be equivalent to Wollan's tree-cut width.

\begin{definition}[GPRTW's tree-cut width]\label{def:tctw-Giannopoulou}
The \emph{GPRTW's width} of a tree-cut decomposition $\T=(T,\X)$ of a graph $G$ is defined as 
\[
	\max\left(\{|\adh(e)|\colon e\in E(T)\}\cup \{w_t \colon t \in V(T)\}
	\vphantom{\widetilde{G}_t}\right).
\]
The \emph{GPRTW's tree-cut width} of $G$ is the minimum width of a tree-cut decomposition of~$G$. 
\end{definition}

\begin{theorem}[\cite{giannopoulou2016linear}]\label{thm:tctw-Wollan-Giannopoulou}
	For every graph $G$, the GPRTW's tree-cut width of $G$ and the Wollan's tree-cut width of $G$ are equal.
\end{theorem}

We point out that the equality described in Theorem~\ref{thm:tctw-Wollan-Giannopoulou} does not apply for every single tree-cut decomposition separately: there are tree-cut decompositions whose Wollan's width and GPRTW's width differ.  
An example can be obtained by taking a complete binary tree of depth $\geq 3$ as the graph $G$, and constructing a tree-cut decomposition $\T$ of $G$ whose underlying tree is a star. The bag of the center of $\T$ consists of the unique vertex of $G$ of degree $2$, while every other vertex of $G$ is placed in a different leaf bag of $\T$. It is easy to verify that $\T$ has Wollan's width $3$, while its GPRTW's width is $\frac{|V(G)|-1}{2}$.
There is no contradiction between this example and the statement of Theorem~\ref{thm:tctw-Wollan-Giannopoulou}: the proof provided in~\cite{giannopoulou2016linear} applies a modification of the given decomposition in order to get the desired bound.

Finally,
let us recall the main result of~\cite{wollan2015structure}: a grid theorem for tree-cut width and immersions. By Theorem~\ref{thm:tctw-Wollan-Giannopoulou}, we may state it equivalently in terms of GPRTW's tree cut-width and in terms of Wollan's tree-cut width.

\begin{theorem}[\cite{wollan2015structure}]\label{thm:grid}
 There exists a function $f\colon \N\to \N$ such that every graph that does not contain a $k\times k$ wall as an immersion has Wollan's (or GPRTW's) tree-cut width at most $f(k)$.
\end{theorem}

Note that if $G$ is a star with $n-1$ leaves and every edge of multiplicity $2$, then it only admits graphs with at most one vertex of degree higher than $2$ as an immersion.
On the other hand, in a star with $n-1$ leaves and every edge of multiplicity $3$, we already find immersion models of all $n$-vertex subcubic graphs (see~\cite[Observation 1]{wollan2015structure}) and in particular, walls.
This gives the intuitive reason behind treating bold and thin edges in tree-cut decompositions differently. 
Indeed, if $\mathcal{T}=(T,\mathcal{X})$ is a tree-cut decomposition of $G$, and $tt'$ is a thin edge of $T$, no two ``internal'' vertices of a wall immersion model in $G$ may lie in different connected components of $T-tt'$.

%% file: objects.tex

In this section we introduce our definition of tree-cut width and relate it to Wollan's definition. Next, we define the dual objects---brambles and tangles---and state our main result.

\subsection{Tree-cut decompositions}

The idea is to use the same notion of tree-cut decompositions, as introduced in Definition~\ref{def:tc-dec}, but to redefine the width. We do it as follows.

Let $\T=(T,\X)$ be a tree-cut decomposition of a graph $G$.
The \emph{adhesion} of a node $t\in V(T)$ is defined as the union of adhesions of bold edges incident to $t$, namely,
\[
	\adh(t)\coloneqq \bigcup \left\{\adh(st)\colon s\textrm{ is such that }st\in E(T)\textrm{ and }|\adh(st)|\geq 3\right\}.
\]
We point that even if an edge $e$ of $G$ participates in the adhesions of two bold edges of $T$ incident to $t$, $e$ is counted only once when computing the cardinality of the adhesion of $t$.
Also note that since only bold edges contribute to the adhesion of a node, for each $t \in V(T)$ we have either $|\adh(t)|=0$ or $|\adh(t)|\geq 3$.
This is illustrated in Figure~\ref{fig:tree}.

	\begin{figure}
	  \center
	  \includegraphics[scale=1.2]{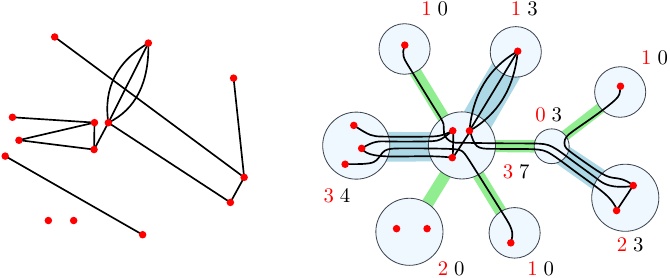}
	  \caption{A tree-cut decomposition (right) of a graph (left). Bold adhesions are depicted by large blue edges, while thin adhesions are thin and green. Near each bag lie two numbers~:~red numbers (first) stand for sizes of bags, and black numbers (second) for sizes of adhesions.}\label{fig:tree}
	\end{figure}

We now present our concept of width. The key aspect is that we will work with two separate width measures, respectively corresponding to the adhesions and to the bags.

\begin{definition}[Adhesion-width and bag-width]\label{def:ab-width}
	Let $\T=(T,\X)$ be a tree-cut decomposition of a graph $G$.
	The {\em{adhesion-width}} of $\T$ and the {\em{bag-width}} of $\T$ are respectively defined as
	$$\adhw(\T)=\max_{t\in V(T)}\,|\adh(t)|\qquad\textrm{and}\qquad \bagw(\T)=\max_{t\in V(T)}\,|X_t|.$$
\end{definition}

Obviously, we can combine the two width measures into one by taking the maximum, and thus we arrive at our proposition for the notion of tree-cut width.

\begin{definition}[ab-tree-cut width]\label{def:tctw-our}
 The {\em{ab-tree-cut width}} of a graph $G$ is the least number $k$ such that $G$ has a tree-cut decomposition with adhesion-width $\leq k$ and bag-width $\leq k$.
\end{definition}

Thus, we have now three notions of tree-cut width: Wollan's tree-cut width (Definition~\ref{def:tctw-Wollan}) and GPRTW's tree-cut width (Definition~\ref{def:tctw-Giannopoulou}) that are equivalent, and the ab-tree-cut width (Definition~\ref{def:tctw-our}).
Let us note that, similarly to Wollan's definition, the new notions of width are closed under taking immersions.

\begin{proposition}\label{prop:tree-cut-immersion}
	Let $G$ be a graph that contains a graph $H$ as an immersion. Suppose $G$ admits a tree-cut decomposition with adhesion-width $\leq a$ and bag-width $\leq b$, for some positive integers~$a,b$. Then so does $H$.
\end{proposition}
\begin{proof}
	Let $\T=(T,\X)$ be a tree-cut decomposition of $G$ with $\adhw(\T)\leq a$ and $\bagw(\T)\leq b$. 
	Fix an immersion model $\pi$ of $H$ in $G$ and for each $t\in V(T)$ define $Y_t\coloneqq\pi^{-1}(X_t)\subseteq V(H)$ to be the set of vertices of $H$ whose images are contained in the bag $X_t$. 
	We prove that $\T'\coloneqq (T,\{Y_t\}_{t\in V(T)})$ is a tree-cut decomposition of $H$ with the desired properties.

	Note that $\{Y_t\}_{t\in V(T)}$ is a near-partition of $V(H)$. 
	As $\pi$ maps vertices of $H$ injectively, we have $|Y_t|\leq |X_t|$ for every $t\in V(T)$, which implies $\bagw(\T')\leq \bagw(\T)$. 
	Observe that the trace in $\T'$ of an edge $e\in E(H)$ with endpoints $u$ and $v$ is contained in the union of traces in $\T$ of edges of the path $\pi(e)$. 
	Moreover, for every edge $st\in E(T)$, if $e$ is an edge of $H$ in $\adh_{\T'}(st)$ with endpoint $u$ and $v$, then there is at least one edge $f$ on the path $\pi(e)$ with $f\in\adh_{\T}(st)$. 
	
	Consider a partial function $\eta\colon E(G)\rightharpoonup E(H)$ defined as follows: if $f\in E(G)$ belongs to $\pi(e)$ for some $e\in E(H)$, then we set $\eta(f)\coloneqq e$, and otherwise $f$ is not in the domain of $\eta$. Note that the validity of this definition is asserted by the fact that the paths in $\{\pi(e)\colon e\in E(G)\}$ are pairwise edge-disjoint. With this notion in place, the observations from the previous paragraph show that 
	$$\adh_{\T'}(st)\subseteq \eta\left(\adh_{\T}(st)\right)\qquad \textrm{for every }st\in E(T).$$
	This implies that
	$$\adh_{\T'}(t)\subseteq \eta\left(\adh_{\T}(t)\right)\qquad \textrm{for every }t\in V(T).$$
	It follows that $\adhw(\T')\leq \adhw(\T)$, as required.
\end{proof}

\subsection{Comparison with original tree-cut width}

We now verify that our new definition is functionally equivalent to the one of Wollan. Formally, we prove the following statement which involves the adjusted definition of tree-cut width proposed by Giannopoulou et al.

\begin{theorem}\label{thm:wollan-comparison}
 Let $G$ be a graph and $\T$ be a tree-cut decomposition of $G$. Then:
 \begin{itemize}[nosep]
  \item If $\T$ has GPRTW's width $k$, then $\T$ has adhesion-width $\leq k^2$ and bag-width $\leq k$.
  \item If $\T$ has adhesion-width $a$ and bag-width $b$, then $\T$ has GPRTW's width $\leq a+b+2$.
 \end{itemize}
\end{theorem}
\begin{proof}
	We show the first implication.
	Let us assume that $\T$ has GPRTW's width $k$.
	In particular, we have $\max\{|\adh(e)|\colon e\in E(T)\}\leq k$ and $\max\{w_t \colon t \in V(T)\}\leq k$.
	Since by definition $w_t \geq |X_t|$ for any $t\in V(T)$, we immediately obtain that $\T$ has bag-width $\leq k$.
	Concerning the adhesion-width, note that $|\{t'\in N_T(t) \colon |\adh(tt')|\geq 3\}|\leq k$ for every ${t\in V(T)}$.
	Since in addition $|\adh(tt')|\leq k$ for every $tt'\in E(T)$, we get that $\T$ has adhesion-width $\leq k^2$.

	We now show the other implication.
	Suppose that $\T$ has adhesion-width $a$ and bag-width~$b$.
	If $a=0$ then $\T$ only contains thin adhesions, and we deduce that GPRTW's width is bounded by the size of a largest bag in $\T$, plus the size of a largest thin adhesion in $\T$ (as the size of thin adhesions actually appears in Definition~\ref{def:tctw-Giannopoulou}).
	Hence $\T$ has GPRTW's width $\leq b+2$ in that case.
	Otherwise by definition, $a\geq 3$.
	Clearly, $|\adh(e)|\leq a$ for every $e\in E(T)$ in that case.
	Since for every $t\in V(T)$, $w_t$ is bounded by the size of $X_t$ plus the size of the adhesion of $t$, we conclude that GPRTW's width is bounded by the maximum of the following two quantities: the largest size of an adhesion of an edge in $\T$, and the largest cumulated size of $X_t$ and the adhesion of $t$, for a node $t$ in $\T$.
	This is bounded by $\max(a, a+b)$, hence the conclusion.
\end{proof}

By combining Theorem~\ref{thm:wollan-comparison} and~\ref{thm:tctw-Wollan-Giannopoulou}, we immediately get the following.

\begin{corollary}\label{cor:equivalent}
 Let $G$ be a graph and let $k$ and $\ell$ be the Wollan's tree-cut width and the ab-tree-cut width of $G$, respectively. Then
 \[
 	\frac{k}{2}-1\leq \ell\leq k^2.
 \]
\end{corollary}

From Corollary~\ref{cor:equivalent} we infer that Theorem~\ref{thm:grid} is also true for ab-tree-cut width.

\subsection{Brambles}

We now move to the first definition of a dual object: a bramble. First, we introduce {\em{slabs}}, which are elements from which the brambles are composed.

\begin{definition}[Slabs]
	A \emph{slab} in a graph $G$ is a pair $(H,K)$ where $H$ is a connected subgraph of $G$ and $K$, called the {\em{core}} of the slab, is a non-empty subset of vertices of $H$ that is $3$-edge-connected in~$G$. Two slabs $(H,K)$ and $(H',K')$ {\em{touch}} if they intersect on their cores, that is, $K\cap K'\neq \emptyset$.
\end{definition}

With slabs in place, brambles are defined as follows.

\begin{definition}[Brambles]
	A \emph{bramble} in a graph is a family of pairwise touching slabs.
\end{definition}

Note that since $3$-edge-connectedness is transitive, all cores of all the slabs in a bramble must be contained in a single $3$-connected component of the graph $G$.

Next, we need to define the order of a bramble. Mirroring the situation in Definition~\ref{def:ab-width}, there will be two notions of an order: one corresponding to adhesions and one corresponding to the bags. For the first one, we need an appropriate notion of hitting a slab.

\begin{definition}[Disconnecting sets]\label{def:disc-set}
	Let $G$ be a graph and $(H,K)$ be a slab in $G$. A set of edges $F\subseteq E(G)$ {\em{disconnects}} $(H,K)$ if there are two vertices $u,v\in K$ that are disconnected by $F\cap E(H)$ in $H$. Further, $F$ is a {\em{disconnecting set}} for a bramble $\B$ if $F$ disconnects every slab in $\B$.
\end{definition}

Note that if the core $K$ of a slab $(H,K)$ has size $1$, then there is no edge subset that disconnects $(H,K)$. We now proceed with defining the order(s) of a bramble.

\begin{definition}[Orders of a bramble]
    Let $\B$ be a bramble in a graph $G$.
	The \emph{adhesion-order} of $\B$
	is the minimum size of a disconnecting set for $\B$, or $+\infty$ if no such disconnecting set exists.
	The \emph{bag-order} of $\B$ is defined as $\min_{(H,K)\in \B} |K|$, or $+\infty$ if $\B$ is empty. The {\em{order}} of $\B$ is the minimum of the adhesion-order and the bag-order of $\B$.
\end{definition}

\subsection{Tangles}

We now move to the second definition of a dual object: a tangle. For this, we need to take a closer look at (edge) separations. The following definitions are essentially taken from the presentation in the book of Diestel~\cite{diestel2012graph}, which in turns cites the work of Diestel and Oum~\cite{diestel2021tangle} as the source of inspiration.

Let us fix a graph $G$.
Consider a separation $\{A,B\}$ of $G$. With such a separation we can associate two {\em{oriented separations}} $(A,B)$ and $(B,A)$.
We say that the oriented separation $(A,B)$ \emph{points toward} $B$.
For a set $\S$ of separations, define $\vec{\S}\coloneqq \{(A,B),(B,A) \colon \{A,B\}\in \S\}$ as the set of all oriented separations associated with the elements of $S$.
An \emph{orientation} of a set $\S$ of separations is a subset $\Ll\subseteq \vec{S}$ which contains precisely one oriented separation associated with every element of $\S$. 
We say that $\Ll$ \emph{avoids} a collection $\Sigma$ of sets of oriented separations if no subset of $\Ll$ belongs to $\Sigma$.

A set $\sigma$ of oriented separations is \emph{consistent} if it does not simultaneously contain separations $(A,B)$ and $(C,D)$ such that $B\cap D=\emptyset$.
Intuitively, this means that there are no two separations in $\sigma$ that ``point away'' from each other.
A non-empty consistent set of oriented separations $\sigma$ is a \emph{star} if $A\cap C=\emptyset$ for all distinct $(A,B),(C,D)\in \sigma$.
A star is illustrated in Figure~\ref{fig:star}.

\begin{figure}
  \center
  \includegraphics[scale=1.1]{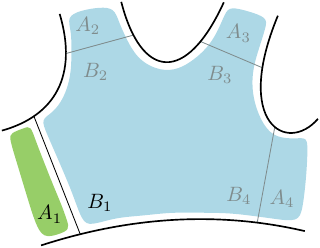}
  \caption{The separations $(A_1,B_1),\dots,(A_4,B_4)$ form a star; $A_1$ is in green and $B_1$ in blue.}\label{fig:star}
\end{figure}

For a positive integer $a$, we let $\S_a$ be the set of all separations of order $<a$ in $G$. Further, for a positive integer $b$, we let $\Sigma_{a,b}$ be the set of all stars $\sigma\subseteq \vec{\S}_a$ satisfying the following conditions:
\begin{align*}
	\left|\bigcup \{ \delta(A) \colon (A,B)\in \sigma \text{ and } |\delta(A)|\geq 3\}\right|&<a,\\
	\left|\bigcap \{ B \colon (A,B)\in \sigma\}\right|&<b.
\end{align*}
We can now present the definition of a tangle.

\begin{definition}[Tangles]
	For a pair of positive integers $a$ and $b$, an \emph{$(a,b)$-tangle} is a consistent orientation of $\S_a$ that avoids~$\Sigma_{a,b}$. We will say that an $(a,b)$-tangle has {\em{adhesion-order}} $a$, {\em{bag-order}} $b$, and {\em{order}} $\min(a,b)$.
\end{definition}
%

\subsection{Main result}

All the pieces are now set and we can state our main result.

\begin{theorem}\label{thm:trinity}
	For any graph $G$ and positive integers $a$ and $b$, the following are equivalent:
	\begin{enumerate}[label=(A\arabic*),ref=(A\arabic*),leftmargin=*]
		\item\label{s:bramble} $G$ has a bramble of adhesion-order $\geq a$ and bag-order $\geq b$;
		\item\label{s:tangle} $G$ has a tangle of adhesion-order $\geq a$ and bag-order $\geq b$;
		\item\label{s:decomp} $G$ has no tree-cut decomposition of adhesion-width $<a$ and bag-width $<b$.
	\end{enumerate}
\end{theorem}

When cast to the variants of definitions with a single parameter, Theorem~\ref{thm:trinity} takes the following form.

\begin{corollary}
	For any graph $G$ and a positive integer $k$, the following are equivalent:
	\begin{enumerate}[label=(B\arabic*),ref=(B\arabic*),leftmargin=*]
		\item $G$ has a bramble of order $\geq k$;
		\item $G$ has a tangle of order $\geq k$;
		\item $G$ has ab-tree-cut width $\geq k$.
	\end{enumerate}
\end{corollary}

%% file: proof.tex

In this section we prove Theorem~\ref{thm:trinity}.
The outline is as follows.
First, we verify that all the three statements always hold for $b=1$, which proves the theorem in this case. Hence, for the remainder of the proof we may assume that $b\geq 2$. Then, in successive subsections we prove implications \ref{s:bramble}$\Rightarrow$\ref{s:tangle}, \ref{s:tangle}$\Rightarrow$\ref{s:decomp}, and \ref{s:decomp}$\Rightarrow$\ref{s:bramble}, in this order. The first two implications are rather straightforward, while the main weight of the proof lies in the last implication. In particular, the subsection containing its proof is preceded by an analysis of the structure of $3$-edge-connected components in an arbitrary graph.

\begin{proof}[Proof of Theorem~\ref{thm:trinity} for $b=1$]
    We show that when $b=1$, all three conditions \ref{s:bramble}, \ref{s:tangle}, and~\ref{s:decomp} hold for every positive integer $a$.

    For \ref{s:bramble}, fix some vertex $v$ and consider the bramble consisting of one slab $(H,K)$ where $K=\{v\}$ and $H=G[\{v\}]$.
	This bramble has bag-order $1$ and adhesion-order $+\infty$. 
	Hence \ref{s:bramble} holds for $b=1$ and any positive integer $a$.

	For \ref{s:tangle}, we argue similarly. Let $a$ be any positive integer. Fix any vertex $v$ and consider the orientation $\Ll$ of $\S_a$ defined as follows: for $\{A,B\}\in \S_a$, we include $(A,B)$ in $\Ll$ provided $v\in B$, and otherwise we include $(B,A)$. Clearly $\Ll$ is consistent and avoids $\Sigma_{a,1}$, because every separation in $\Ll$ points towards the side that contains $v$. So $\Ll$ is an $(a,1)$-tangle, as~required.
	
	Finally, \ref{s:decomp} trivially holds for $b=1$ and any positive integer $a$, because $G$ is assumed to be non-empty.
\end{proof}

Therefore, for the remainder of the proof we assume that $b\geq 2$.

\subsection{From a bramble to a tangle}

\begin{proof}[Proof of Theorem~\ref{thm:trinity}, \ref{s:bramble}$\Rightarrow$\ref{s:tangle}]
	Let $\B$ be a bramble of adhesion-order $\geq a$ and bag-order $\geq b$, where $b\geq 2$.
	
	We claim that for each separation $\{A,B\}$ in $\S_a$, exactly one of the sides $A$ and $B$ contains the core $K$ of at least one slab $(H,K)\in \B$ as a subset.
	Indeed, if none of $A$ and $B$ contained such a core, then $\delta(A)$ would be a disconnecting set for $\B$. Since $|\delta(A)|<a$, this is a contradiction.
	On the other hand, if both $A$ and $B$ contained cores of some slabs from $\B$, then these cores would not intersect, again a contradiction.
	
	We construct a tangle $\Ll$ in $G$ by orienting every separation $\{A,B\}\in \S_a$ so that it points to the part that contains the core of at least one slab from $\B$.
	Note that $\Ll$ is consistent, because otherwise $\B$ would contain two slabs with disjoint cores. It remains to prove that $\Ll$ avoids~$\Sigma_{a,b}$.

	For contradiction, suppose that there exists a star $\sigma\subseteq \Ll$ such that 
	\begin{align}
		\big|\bigcup \{ \delta(A) \colon (A,B)\in \sigma \text{ and } |\delta(A)|\geq 3\}\big|&<a,\label{eq:beaver}\\
		\big|\bigcap \{ B \colon (A,B)\in \sigma\}\big|&<b.\label{eq:otter}
	\end{align}
	Let $$F\coloneqq \bigcup \{ \delta(A) \colon (A,B)\in \sigma \text{ and } |\delta(A)|\geq 3\}.$$
	Consider any slab $(H,K)\in \B$.
	Since the bag-order of $\B$ is at least $b$, we have $|K|\geq b$.
	By~\eqref{eq:otter}, there exists a separation $(A,B)\in \sigma$ such that $A$ and $K$ intersect. Note that it cannot happen that $K\subseteq A$, because $(A,B)\in \Ll$, so $A$ cannot fully contain any core of a slab from $\B$. We infer that $\delta(A)$ disconnects the slab $(H,K)$. Since $K$ is $3$-edge-connected in $G$, this in particular implies that $|\delta(A)|\geq 3$. So $\delta(A)\subseteq F$ and $F$ disconnects $(H,K)$. Since the slab $(H,K)$ was chosen arbitrarily from $\B$, we conclude that $F$ is a disconnecting set for $\B$. However, now~\eqref{eq:beaver} stands in contradiction with the assumption that the adhesion-order of $\B$ is at least $a$.
\end{proof}

\subsection{From a tangle to the non-existence of a decomposition}

\begin{proof}[Proof of Theorem~\ref{thm:trinity}, \ref{s:tangle}$\Rightarrow$\ref{s:decomp}]
	Suppose toward a contradiction that $G$ admits a tree-cut decomposition $\T=(T,\X)$ with $\adhw(\T)<a$ and $\bagw(\T)<b$, as well as a tangle $\Ll$ with adhesion-order $\geq a$ and bag-order $\geq b$.
	
	First, note that for every edge $ab \in E(T)$ we can define a separation $\{A,B\}$ of $G$, where $A$ consists of all vertices $u$ such that the node of $T$ whose bag contains $u$ lies in the same connected component of $T-ab$ as $a$, and $B=V(G)\setminus A$. Since $\adhw(\T)<a$, it follows that the order of $\{A,B\}$ is smaller than $a$, that is, $\{A,B\}\in \S_a$. 
	Hence, exactly one of $(A,B)$ and $(B,A)$ belongs to~$\Ll$. 
	Define an orientation $\vec{T}$ of the tree $T$ as follows: orient $ab$ towards $b$ if $(A,B) \in \Ll$, and towards $a$ otherwise.
	Since $T$ has less edges than nodes, it follows that in $\vec{T}$ there exists a node $r$ whose outdegree in $\vec{T}$ is zero.
	In other words, every edge incident to $r$ in $T$ points toward $r$ in $\vec{T}$.
	Thus, the set $\sigma$ of oriented separations corresponding to the edges incident to $r$ forms a star in $\S_a$.
	Since $T$ has adhesion-width $<a$ and bag-width $<b$, we immediately get that $\sigma\in \Sigma_{a,b}$. Hence $\Ll$ does not avoid $\Sigma_{a,b}$, a contradiction.
\end{proof}

\subsection{Structure of $3$-edge-connected components}\label{sec:3cc}

Before we proceed to the proof of the last implication, we need to prove some auxiliary results about tree-cut decompositions of a graph and its $3$-edge-connected components.

Consider a graph $G$ and let $\tccrel$ be the equivalence relation on $V(G)$ defined as being in the same $3$-edge-connected component.
We define
$$G_{\tcc}\coloneqq \quo{G}{\tccrel}.$$
Our goal now is to understand the structure of $G_{\tcc}$. First, we observe that $G_{\tcc}$ has no non-trivial $3$-edge-connected components.

\begin{lemma}\label{lem:tcc-idempotent}
 In $G_{\tcc}$, no two different vertices are $3$-edge-connected.
\end{lemma}
\begin{proof}
 Consider any two different vertices $A$ and $B$ of $G_{\tcc}$. Recall that $A$ and $B$ are two different $3$-edge-connected components of $G$, hence let us pick arbitrary vertices $a\in A$ and $b\in B$. Since $a$ and $b$ belong to different $3$-edge-connected components of $G$, there exists a separation $\{A',B'\}$ of $G$ of order $<3$ such that $a\in A'$ and $b\in B'$. Note that since the order of $\{A',B'\}$ is $<3$, every $3$-edge-connected component $C$ of $G$ has to be entirely contained either in $A'$ or in $B'$. In particular, $A\subseteq A'$ and $B\subseteq B'$. Thus, $\{A',B'\}$ naturally induces a separation $\{\widehat{A},\widehat{B}\}$ of $G_{\tcc}$ defined by placing each $C\in V(G_{\tcc})$ in $\widehat{A}$ provided $C\subseteq A'$, and in $\widehat{B}$ provided $C\subseteq B'$. Clearly, $A\in \widehat{A}$, $B\in \widehat{B}$, and the order of $\{\widehat{A},\widehat{B}\}$ is the same as of $\{A',B'\}$. This means that $\{\widehat{A},\widehat{B}\}$ witnesses that $A$ and $B$ are not $3$-edge-connected in $G_{\tcc}$.
\end{proof}

It appears that the structure of graphs satisfying the condition stated in Lemma~\ref{lem:tcc-idempotent} can be nicely described: they are {\em{cacti}}. More precisely, a {\em{cactus}} is a graph where every $2$-(vertex)-connected component is either a single edge or a cycle (here, we allow cycles of length $2$, that is, pairs of parallel edges). We have the following observation, which is essentially (up to technical details in definitions) known in the literature~\cite{dinits1976structure,mehlhorn2017certifying}.

\begin{lemma}\label{lem:cactii-characterization}
 A graph $G$ is a cactus if and only if no two different vertices of $G$ are $3$-edge-connected.
\end{lemma}
\begin{proof}
 Assume first that $G$ is a cactus. Take any two different vertices $u,v$ of $G$ and for contradiction suppose that there are three edge-disjoint paths $P_1,P_2,P_3$ connecting $u$ and $v$. Let $u_1,u_2,u_3$ be the vertices directly succeeding $u$ on $P_1,P_2,P_3$, respectively. Within the closed walk $P_1\cup P_2$ one can find a cycle that contains $u$, $u_1$, and $u_2$, which witnesses that all these three vertices belong to the same $2$-connected component of $G$. The same can be argued about the triple $u$, $u_2$, and $u_3$, so all the four vertices $u$, $u_1$, $u_2$, and $u_3$ belong to the same $2$-connected component of $G$. However, the first edges of $P_1,P_2,P_3$ are pairwise different, which means that $u$ has degree at least $3$ within this $2$-connected component. This contradicts the assumption that $G$ is a cactus.
 
 Assume now that $G$ is not a cactus, which means that there exists a $2$-connected component $C$ of $G$ that is neither a single edge nor a cycle. This means that $C$ has a vertex of degree $3$, say~$u$. Let $e_1,e_2,e_3$ be an arbitrary triple of distinct edges of $C$ incident to $u$, and let $u_1,u_2,u_3$ be the other endpoints of $e_1,e_2,e_3$, respectively. Since $C$ is $2$-connected, $C-u$ is connected, hence in $C-u$ there exists a tree $T$ whose set of leaves is either $\{u_1,u_2,u_3\}$, or consist of two elements among $\{u_1,u_2,u_3\}$ with the third vertex (call it $u'$) lying in a path between the other two in $T$. Now $T$ with vertex $u$ and edges $e_1,e_2,e_3$ added is a graph consisting of $u$, another vertex $v$ of degree $3$ (in the first case: a leaf of $T$ if two of $u_1, u_2, u_3$ are equal, or an internal node of $T$ if they are pairwise distinct; $u'$ in the second case), and $3$ internally vertex-disjoint paths connecting $u$ and $v$. Since this graph is a subgraph of $G$, it follows that $u$ and $v$ are $3$-edge-connected in~$G$.
\end{proof}

By combining Lemmas~\ref{lem:tcc-idempotent} and~\ref{lem:cactii-characterization} we get the following.

\begin{corollary}\label{cor:cactus}
	For any graph $G$, the graph $G_{\tcc}$ is a cactus.
\end{corollary}

We now relate the ab-tree-cut width of a graph $G$ with the ab-tree-cut width of its $3$-edge-connected components. For this, we need to associate with each $3$-edge-connected component $A$ a suitable {\em{torso}} of $A$, which is a graph that reflects connections between vertices of $A$ that are realized either by edges within $A$ or paths that are internally disjoint with $A$. Formally, the torso of $A$ is the graph $\torso(A)$ obtained from $G$ as follows.
Observe that for every connected component $Z$ of $G-A$, there are at most two edges in $G$ having one endpoint in $Z$ and the other in $A$. 
Then, for every such component $Z$,
\begin{itemize}[nosep]
 \item remove $Z$ completely, provided $Z$ has at most one neighbor in $A$; or
 \item remove $Z$ and replace it with a new edge $f_Z$ connecting the neighbors of $Z$ in $A$, provided $Z$ has exactly two neighbors in $A$. 
\end{itemize}
Note that this second operation does not create loops.
The edge $f_Z$ is called the {\em{replacement edge}} of the component $Z$.

We have the following simple observations about the torso operation.

\begin{lemma}\label{lem:torso-immersion}
 For every graph $G$ and a $3$-edge-connected component $A$ of $G$, the graph $G$ contains $\torso(A)$ as an immersion.
\end{lemma}
\begin{proof}
 It suffices to map every vertex of $A$ to itself, every edge of $G[A]$ to itself, and every replacement edge $f_Z$ to any path in $G$ that connects the endpoints of $f_Z$ and has all the internal vertices in $Z$.
\end{proof}

\begin{lemma}\label{lem:torso-tcc}
 For every graph $G$ and a $3$-edge-connected component $A$ of $G$, the graph $\torso(A)$ is $3$-edge-connected.
\end{lemma}
\begin{proof}
 Consider any pair of vertices $u,v\in A$. Since $A$ is $3$-edge-connected in $G$, there are three edge-disjoint paths $P_1,P_2,P_3$ in $G$ that connect $u$ and $v$. These can be naturally projected to paths $Q_1,Q_2,Q_3$ in $\torso(A)$ as follows: for every maximal subpath of $P_i$, $i\in \{1,2,3\}$, whose all internal vertices do not belong to $A$, say they belong to some connected component $Z$ of $G-A$, replace this subpath with the replacement edge $f_Z$. It can be easily seen that each connected component $Z$ of $G-A$ will participate in at most one such replacement, because there are at most two edges connecting $Z$ with $A$. Hence, $Q_1,Q_2,Q_3$ remain edge-disjoint and $u$ and $v$ are $3$-edge-connected in $\torso(A)$.
\end{proof}

The next theorem is the main outcome of this section. Intuitively, it will allow us to focus on a single $3$-edge-connected component of $G$ when constructing a bramble of high order. A statement in the work of Wollan that mirrors this step is~\cite[Lemma~5]{wollan2015structure}.

\begin{theorem}\label{thm:torsos-width}
	Let $G$ be a graph and $a,b$ be two positive integers.
	Then $G$ admits a tree-cut decomposition of adhesion-width $\leq a$ and bag-width $\leq b$ if and only if for every $3$-edge-connected component $A$ of $G$, graph $\torso(A)$ admits a tree-cut decomposition of adhesion-width $\leq a$ and bag-width $\leq b$.
\end{theorem}
\begin{proof}
 The forward implication follows by combining Lemma~\ref{lem:torso-immersion} with Proposition~\ref{prop:tree-cut-immersion}. We are left with proving the backward implication.
 
 Assume then that for every $3$-edge-connected component $A$ of $G$, there is a tree-cut decomposition $\T^A=(T^A,\X^A)$ of $\torso(A)$ such that $\adhw(\T^A)\leq a$ and $\bagw(\T^A)\leq b$. The goal is to ``glue'' the decompositions $\T^A$ into a single tree-cut decomposition $\T$ of $G$ so that the guarantees about the width measures are preserved. The gluing will be done along the graph $G_{\tcc}$, which by Corollary~\ref{cor:cactus} is a cactus.
 
 We execute the gluing as follows. Fix any spanning forest of $G_{\tcc}$ and let $S$ be its edge set. Further, let $R=E(G_{\tcc})\setminus S$ be the remaining edges of $G_{\tcc}$. Note that $R$ contains exactly one edge from every $2$-connected component $C$ of $G_{\tcc}$ that is a cycle (that is, is not a single edge). Call this edge $r_C$.
 
 Recall that in the construction of $G_{\tcc}=\quo{G}{\tccrel}$, every edge of $G_{\tcc}$ originates in some edge of $G$ that connects two different $3$-edge-connected components. Let $\alpha\colon E(G_{\tcc})\to E(G)$ be this origin mapping: $\alpha(e)$ is the edge of $G$ from which $e$ originates. Note that $\alpha$ is injective.
 
 We construct a forest $T$ from trees $T^A$ as follows. First, take the disjoint union of trees~$T^A$. Then, consider every edge $e\in S$, say with endpoints $A,B\in V(G_{\tcc})$. Let $u$ and $v$ be the endpoints of $\alpha(e)$, where $u\in A$ and $v\in B$. Clearly, there are nodes $p\in V(T^A)$ and $q\in V(T^B)$ such that $u\in X^A_p$ and $v\in X^B_q$. Then add the edge $pq$ to the forest $T$, and call this edge $\gamma(e)$. This concludes the construction of $T$. It is easy to see that since $S$ is the edge set of a forest and each $T^A$ is a tree, $T$ is also a forest.
 
 We now associate the nodes of $T$ with bags $\X=\{X_s\colon s\in V(T)\}$ inherited from decomposition $\T^A$ in the natural manner: if a node $s$ of $T$ originates from the tree $T^A$, then we set $X_s\coloneqq X^A_s$. Since $3$-edge-connected components of $G$ form a partition of $V(G)$, it follows that $\X$ is a near-partition of $V(G)$. Thus, $\T\coloneqq (T,\X)$ is\footnote{Formally, we required tree-cut decompositions to be trees and not just forests, but we can always add arbitrary edges to make $T$ connected without increasing any of the width measures.} a tree-cut decomposition of $G$.
 
 It remains to analyze the width measures of $\T$. Since the bags are directly taken from decompositions $\T^A$, which have bag-width $\leq b$, we immediately see that $\bagw(\T)\leq b$. The argument for the adhesion-width is a bit more complicated, because the adhesions may actually change during gluing. 
 
 We first show the following claim about connectedness in $G_{\tcc}$ and in~$G$.
 
 \begin{claim}\label{cl:conn-tcc}
  Let $A$ be a $3$-edge-connected component of $G$ and let $C$ be a connected component of $G_{\tcc}-A$. Then the vertices of $\bigcup_{D\in V(C)} D$ lie in the same connected component of $G-A$.
 \end{claim}
 \begin{proof}
 Let $\widehat{C}\coloneqq \bigcup_{D\in V(C)} D$.
 Suppose for contradiction that there is a partition $\{L,R\}$ of $\widehat{C}$ such that in $G-A$ there is no edge with one endpoint in $L$ and second in $R$, and $L$, $R$ are non-empty. Since $C$ is connected in $G_{\tcc}$, there must exist a $3$-edge-connected component $B\in V(C)$ such that both $L\cap B$ and $R\cap B$ are non-empty, as otherwise we may find two adjacent $B,B'\in V(C)$ with $B\subseteq L$, $B'\subseteq R$, hence an edge in $G$ having one endpoint in $L$ and the other in $R$. Pick any $u\in L\cap B$ and $v\in R\cap B$. Since $B$ is $3$-edge-connected in $G$, there are three edge-disjoint paths $P_1,P_2,P_3$ in $G$ connecting $u$ and $v$. Note that in $G$ there are at most $2$ edges with one endpoint in $\widehat{C}$ and the other outside of $\widehat{C}$: these are the images of the at most two edges between $A$ and $C$ under $\alpha$. At most two of the paths $P_1,P_2,P_3$ can contain any of these at most two edges, hence one of them, say $P_1$, must have all the vertices contained in $\widehat{C}$. But then $P_1$ contains an edge with one endpoint in $L$ and second in $R$, a contradiction.
 \cqed\end{proof}

 Claim~\ref{cl:conn-tcc} provides us with an understanding of the replacement edges in torsos of the $3$-edge-connected components of $G$. Precisely, let $A$ be a $3$-edge-connected component of $G$ and let $C$ be any $2$-connected component of $G_{\tcc}$ that is a cycle and contains $A$. By Claim~\ref{cl:conn-tcc}, all the vertices of $\bigcup_{D\in V(C)\setminus \{A\}} V(D)$ lie in the same connected component of $G-A$. Call this component $Z^A(C)$. 
 Note that since $C$ ranges over all $2$-connected components of $G_{\tcc}$ that are cycles containing $A$, the components $Z^A(C)$ are pairwise different. Finally, if $e^1,e^2$ are the two edges of $C$ that are incident to $A$, then the replacement edge of $Z^A(C)$ in $\torso(A)$ connects the endpoints of $\alpha(e^1)$ and $\alpha(e^2)$ that lie in $A$, or is non-existent if these endpoints coincide. 
 
 Using all these observations we can understand the traces of edges of $G$ in the decomposition $\T$. 
 For an edge $e$ of $G$, by $\trace_{\T}(e)$ we mean the edge set of the trace of $e$ in $\T$. Similarly, if $e$ is an edge of $\torso(A)$, for some $3$-edge-connected component $A$, then $\trace_{\T^A}(e)$ is the edge set of the trace of $e$ in $\T^A$. The following claim explains how the traces of the edges of $G$ behave in $\T$. The proof is a straightforward check using the observations presented above, hence we omit it.
 It can be followed on Figure~\ref{fig:cactus}.

 \bigskip 

\begin{figure}
  \center
  \includegraphics[scale=1.2]{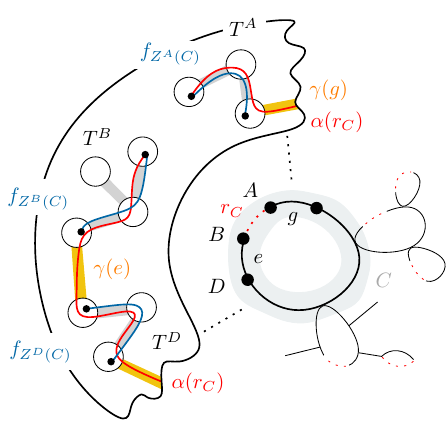}
  \caption{A close look at the trace of $\alpha(r_C)$ in the proof of Theorem~\ref{thm:torsos-width}.}\label{fig:cactus}
\end{figure}
 
 \begin{claim}\label{cl:traces}
  Let $e$ be an edge of $G$ with endpoints $u$ and $v$, and let $A$ and $B$ be the $3$-edge-connected components of $G$ such that $u\in A$ and $v\in B$. Then:
  \begin{itemize}[nosep]
   \item If $A=B$, then $$\trace_\T(e)=\trace_{\T^A}(e).$$
   \item If $A\neq B$ and $\alpha^{-1}(e)\in S$, then $$\trace_\T(e)=\{\gamma(\alpha^{-1}(e))\}.$$
   \item If $A\neq B$ and $\alpha^{-1}(e)\in R$, say $\alpha^{-1}(e)=r_C$ for some $2$-connected component $C$ of $G_{\tcc}$ that is a cycle, then 
   $$\trace_{\T}(e)=\bigcup_{D\in V(C)} \trace_{\T^D}\left(f_{Z^{D}(C)}\right)\cup \bigcup_{g\in E(C)\setminus \{e\}} \{\gamma(g)\}.$$
   Here, if the replacement edge $f_{Z^{D}(C)}$ does not exist, we take $\emptyset$ for the corresponding trace.
  \end{itemize}
 \end{claim}

 For a $3$-edge-connected component $A$ we define a mapping $\eta^A\colon E(\torso(A))\to E(G)$ as follows:
 \begin{itemize}[nosep]
  \item If $e$ is not a replacement edge, then set $\eta^A(e)=e$.
  \item If $e$ is a replacement edge, say $e=f_{Z}$ for some connected component $Z$ of $G-A$, then observe that there is a unique $2$-connected component $C$ of $G_{\tcc}$ that is a cycle containing~$A$ and for which $Z=Z^A(C)$. Then set $\eta^A(e)=\alpha(r_C)$.
 \end{itemize}
 With this notation in place, Claim~\ref{cl:traces} immediately gives the following characterization of adhesions in $\T$.
 
 \begin{claim}\label{cl:adhesions}
  Let $e$ be an edge of $T$. Then:
  \begin{itemize}[nosep]
  \item If $e=\gamma(g)$ for some edge $g\in S$, say belonging to a $2$-connected component $C$ of $G_{\tcc}$ that is a cycle, then
  $$\adh_{\T}(e)=\{\alpha(g),\alpha(r_C)\}.$$
  \item If $e\in E(T^A)$ for some $3$-edge-connected component $A$, then
  $$\adh_{\T}(e)=\eta^A\left(\adh_{\T^A}(e)\right).$$
  \end{itemize}
 \end{claim}

 From the first point of Claim~\ref{cl:adhesions} it follows that all edges of $T$ that originate from the spanning forest $S$ are thin in $\T$, hence they do not contribute to the adhesions of the nodes of $\T$. Then, from the second point of Claim~\ref{cl:adhesions} we observe that $|\adh_{\T}(s)|\leq |\adh_{\T^A}(s)|$ for every node $s$ of $T$ that originates from $T^A$. Since $\adhw(\T^A)\leq a$ for every $3$-edge-connected component $A$, it follows that $\adhw(\T)\leq a$.
\end{proof}

\subsection{From the non-existence of a decomposition to a bramble}

We are now ready to prove the last implication of Theorem~\ref{thm:trinity}. The proof closely follows the line of argumentation for the treewidth case presented by Diestel in~\cite{diestel2012graph}, which in turn is based on a proof by Mazoit~\cite{mazoit-brambles}.

\begin{proof}[Proof of Theorem~\ref{thm:trinity}, \ref{s:decomp}$\Rightarrow$\ref{s:bramble}]
	Assume that $G$ has no tree-cut decomposition of adhesion-width $<a$ and bag-width $<b$.
	We deduce by Theorem~\ref{thm:torsos-width} that there exists a $3$-edge-connected component $A$ of $G$ such that every tree-cut decomposition $\T$ of $\torso(A)$ satisfies at least one of the conditions: $\adhw(\T)\geq a$ or $\bagw(\T)\geq b$.
	We will construct a suitable bramble using the component $A$.
	Denote $G_A\coloneqq \torso(A)$ for brevity.

	A tree-cut decomposition $\T=(T,\X)$ of $G_A$ shall be called \emph{good} if for every node $t$, we~have 
	\begin{align*}
		&\adh_\T(t)<a, \textrm{and}\\
		&\textrm{if $|X_t|\geq b$, then $t$ is a leaf of $T$.}
	\end{align*}
	In other words, a good tree-cut decomposition has adhesion-width $<a$, and the only nodes whose bags are allowed to be of size $\geq b$ are the leaves. 
	Clearly, there always exists a good tree-cut decomposition of $G_A$, e.g., the one consisting of a single node, whose bag contains all the vertices of $G_A$.
	If $\T=(T,\X)$ is a good tree-cut decomposition and a leaf $t$ of $T$ satisfies $|X_t|\geq b$, then $X_t$ shall be called a \emph{petal} of $\T$.
	Note that the assumption that $G_A$ has no tree-cut decomposition of adhesion-width $<a$ and bag-width $<b$ implies that every good tree-cut decomposition of $G_A$ has a petal.
	
	Our first goal is to construct a bramble $\B_A$ of adhesion-order $\geq a$ and bag-order $\geq b$ in~$G_A$. For this, let $\M\subseteq 2^A$ be the family of all petals of all good tree-cut decompositions of $G_A$. Further, let $\F\subseteq \M$ be an inclusion-wise minimal subfamily of $\M$ satisfying the following two conditions:
	\begin{enumerate}[label=(\roman*),ref=(\roman*),leftmargin=*,nosep]
	 \item\label{p:hits} For each good tree-cut decomposition $\T$ of $G_A$, $\F$ contains at least one petal of $\T$.
	 \item\label{p:upward} $\F$ is {\em{upward-closed}}: if $C,D\in \M$ are such that $C\subseteq D$ and $C\in \F$, then also $D\in \F$.
	\end{enumerate}
    We observe the following.
    
    \begin{claim}\label{cl:witness}
     Suppose $C$ is an inclusion-wise minimal element of $\F$. 
     Then there exists a good tree-cut decomposition $\T^C$ of $G_A$ such that $C$ is a petal of $\T^C$, and moreover $C$ is the only petal of $\T^C$ that belongs to $\F$.
    \end{claim}
    \begin{proof}
     Observe that since $C$ is an inclusion-wise minimal element of $\F$, the family $\F\setminus \{C\}$ is upward-closed, that is, satisfies~\ref{p:upward}. As $\F$ is inclusion-wise minimal subject to satisfying both~\ref{p:upward} and~\ref{p:hits}, it follows that $\F\setminus \{C\}$ does not satisfy~\ref{p:hits}, which implies the claim.
    \cqed\end{proof}

    The next claim is the key observation.
	
	\begin{claim}\label{cl:bramble}
	 The elements of $\F$ pairwise intersect.
	\end{claim}
	\begin{proof}
	 Suppose otherwise: there exist sets $C,D\in \F$ that are disjoint. By possibly replacing each of $C$ and $D$ by its subset, we may assume that $C$ and $D$ are inclusion-wise minimal elements of $\F$. By Claim~\ref{cl:witness}, there exist good tree-cut decompositions $\T^C=(T^C,\X^C)$ and $\T^D=(T^D,\X^D)$ of $G_A$ such that $C$ is the only petal of $\T^C$ that belongs to $\F$ and $D$ is the only petal of $\T^D$ that belongs to $\F$. Let $t^C$ be the leaf of $T^C$ whose bag is $C$, and define $t^D$ analogously.
	 
	 Let $\Pp$ be a maximum-size family of edge-disjoint paths in $G_A$ connecting $C$ with $D$. By Menger's theorem, there exists a separation $\{\wh{C},\wh{D}\}$ of $G_A$ such that $|\delta(\wh{C})|=|\delta(\wh{D})|=|\Pp|$. In particular, every path $P\in \Pp$ contains exactly one edge in $\delta(\wh{C})$, all the vertices on $P$ before this edge belong to $\wh{C}$, and all the vertices on $P$ after this edge belong to $\wh{D}$.
	This is illustrated in Figure~\ref{fig:menger}.

	\begin{figure}
	  \center
	  \includegraphics[scale=1.1]{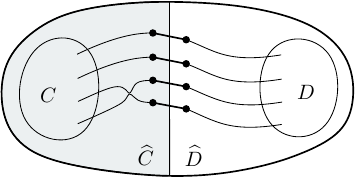}
	  \caption{A separation $\{\wh{C},\wh{D}\}$ of $G_A$ such that $|\delta(\wh{C})|=|\delta(\wh{D})|=|\Pp|$. Edges in $\delta(\wh{C})$ are depicted in bold.}\label{fig:menger}
	\end{figure}

	 We now construct a decomposition $\T=(T,\X)$ of $G_A$ as follows. Construct $T$ by taking the disjoint union of $T^C$ and $T^D$, removing the nodes $t^C$ (from $T^C$) and $t^D$ (from $T^D$), and adding the edge $s^Cs^D$, where $s^C$ is the unique neighbor of $t^C$ in $T^C$ and $s^D$ is the unique neighbor of $t^D$ in $T^D$. (Note here that $s^C$ and $s^D$ exist, as otherwise either $C=A$ or $D=A$, implying that the other one is empty, but every element of $\F$ has size at least $b\geq 2$.) Next, define the bags $\X=\{X_t\colon t\in V(T)\}$ as follows: for every node $t$ originating from $T^C$ we set $X_t\coloneqq X^C_t\cap \widehat{D}$, and for every node $t$ originating from $T^D$ we set $X_t\coloneqq X^D_t\cap \widehat{C}$. Thus, $\{X_t\colon t\in V(T^C)\setminus \{t^C\}\}$ is a near-partition of $\widehat{D}$ and $\{X_t\colon t\in V(T^D)\setminus \{t^D\}\}$ is a near-partition of $\widehat{C}$, implying that $\X$ is a near-partition of $A$. So $\T$ is a tree-cut decomposition of $G_A$.
	 
	 We now bound the adhesion-width of $\T$. Similarly as before, for an edge $e$, the edge set of the trace of $e$ in $\T$ is denoted by $\trace_\T(e)$; similarly for decompositions $\T^C$ and $\T^D$.
     Observe that for every edge $e$ of $G$, say with endpoints $u$ and $v$, the trace of $e$ in $\T$ can be characterized as follows.
	 \begin{itemize}[nosep]
	  \item If $u,v\in \wh{C}$, then $\trace_{\T}(e)=\trace_{\T^D}(e)$.
	  \item If $u,v\in \wh{D}$, then $\trace_{\T}(e)=\trace_{\T^C}(e)$.
	  \item Suppose $u\in \wh{C}$ and $v\in \wh{D}$. Then there exists a path $P\in \Pp$ such that $e$ lies on $P$. Let $P^C$ be the prefix of $P$ consisting of edges with at least one endpoint in $\wh{C}$, and $P^D$ be the suffix of $P$ consisting of edges with at least one endpoint in $\wh{D}$ (thus, $E(P^C)\cap E(P^D)=\{e\}$). Then
	 \end{itemize}
	$$\trace_{\T}(e)\subseteq \left(\bigcup_{f\in E(P^D)}\trace_{\T^C}(f)\setminus \{t^Cs^C\}\right)\cup \left(\bigcup_{f\in E(P^C)}\trace_{\T^D}(f)\setminus \{t^Ds^D\}\right)\cup \{s^Cs^D\}.$$
    Let $\eta^C\colon E(G_A)\rightharpoonup E(G_A)$ be a partial function defined as follows:
    \begin{itemize}[nosep]
     \item each edge $e$ with both endpoints in $\wh{C}$ is mapped to itself;
     \item each other edge $e$ is not in the domain of $\eta^C$, unless it belongs to some path $P\in \Pp$, in which case it is mapped to the unique edge of $P$ with one endpoint in $\wh{C}$ and the other in $\wh{D}$.
    \end{itemize}
    Define a partial function $\eta^D\colon E(G_A)\rightharpoonup E(G_A)$ symmetrically using $\wh{D}$ instead of $\wh{C}$. Then from the above characterization of traces it follows that:
    \begin{itemize}[nosep]
     \item For each node $t\in V(T^C)$, we have
     $$\adh_{\T}(t)\subseteq \eta^D\left(\adh_{\T^C}(t)\right).$$
     \item For each node $t\in V(T^D)$, we have
     $$\adh_{\T}(t)\subseteq \eta^C\left(\adh_{\T^D}(t)\right).$$
    \end{itemize}
    Since both $\T^C$ and $\T^D$ have adhesion-width $<a$, it follows that $\adhw(\T)<a$.
    
    Finally, observe that every bag in $\T$ is a subset of a bag originating either from $\T^C$ or from~$\T^D$. Since non-leaf nodes of $\T$ originate from non-leaf nodes of $\T^C$ and $\T^D$, it follows that only the leaves of $\T$ may have bags of size $\geq b$, hence $\T$ is good. Further, for every leaf $t$ of $\T$, the bag at $t$ in $\T$ is either a subset of a leaf bag in $\T^C$ other than $C$, or a subset of a leaf bag in $\T^D$ other than $D$. Since $C$ and $D$ are the only petals of $\T^C$ and $\T^D$, respectively, that belong to $\F$, and $\F$ is upward-closed, it follows that $\T$ has no petals that belong to $\F$. This is a contradiction with property~\ref{p:hits} of $\F$.
	\cqed\end{proof}

    We can now define the bramble $\B_A$ as follows: for each connected subgraph $H$ of $G_A$ such that $V(H)\in \F$, include the slab $(H,V(H))$ in $\B_A$. Note that since $G_A$ is $3$-edge-connected by Lemma~\ref{lem:torso-tcc}, these are indeed slabs in $G_A$. Further, Claim~\ref{cl:bramble} shows that $\B_A$ is a bramble. We now verify the orders of $\B_A$.
    
    \begin{claim}
     $\B_A$ has adhesion-order $\geq a$.
    \end{claim}
    \begin{proof}
     Consider any set of edges $F\subseteq E(G_A)$ satisfying $|F|<a$. Construct a tree-cut decomposition $\T^F$ of $G_A$ as follows. For each connected component $D$ of $G_A-F$, construct a node $t^D$ with $V(D)$ as its bag. Finally, construct a root node $r$ with an empty bag and make it adjacent to all the nodes $t^D$. (Thus, the tree underlying $\T^F$ is a star with $r$ being the center.) Since the adhesions of all the nodes are contained in $F$, it follows that $\adhw(\T^F)<a$. Further, every node other than $r$ is a leaf, and $r$'s bag is empty, so we conclude that $\T^F$ is a good tree-cut decomposition of $G_A$. By property~\ref{p:hits} of $\F$, there is a connected component $D$ of $G_A-F$ such that $V(D)\in \F$. Observe that the slab $(D,V(D))$ has been included in $\B_A$ and $E(D)\cap F=\emptyset$, so in particular this slab is not disconnected by $F$. Since $F$ was chosen arbitrarily, we conclude that $\B_A$ has no disconnecting set of size $<a$. 
    \cqed\end{proof}

    \begin{claim}
     $\B_A$ has bag-order $\geq b$.
    \end{claim}
    \begin{proof}
     It suffices to note that every element of $\F$ is a petal of some good tree-cut decomposition of $G_A$, and hence has size $\geq b$.
    \cqed\end{proof}

    Now that the bramble $\B_A$ in $G_A$ is constructed, we can modify it to obtain a bramble $\B$ in $G$. Consider any slab $(H,K)\in \B_A$ and recalling that $H$ is a subgraph of $G_A$, construct a subgraph $H'$ of $G$ from $H$ as follows: for every replacement edge present in $H$, say edge $f_Z$ for some connected component $Z$ of $G-A$, replace $f_Z$ by an arbitrary path connecting the endpoints of $f_Z$ that has all internal vertices in $Z$. Note that $H'$ remains connected and $K$, as a subset of $A$, is $3$-edge-connected in $G$. Hence $(H',K)$ is a slab in $G$. We define $\B$ as the set of all slabs $(H',K)$ obtained from slabs $(H,K)\in \B_A$ as described above.
    
    Since the cores of slabs did not change, $\B$ is a bramble in $G$ and its bag-order is the same as that of $\B_A$, which is $\geq b$. To see that the adhesion-order of $\B$ is not smaller than that of $\B_A$, observe the following: if $F'\subseteq E(G)$ is a disconnecting set for $\B$, then replacing every edge of $F'$ with an endpoint outside of $A$, say in a component $Z$ of $G-A$, with the replacement edge~$f_Z$, turns $F'$ into a disconnecting set $F$ for $\B$ such that $|F|\leq |F'|$.
	Therefore, we conclude that $\B$ is a bramble in $G$ of adhesion-order $\geq a$ and bag-order $\geq b$.
\end{proof}

\section{Cops, dogs, and robber game}\label{sec:game}

In this section we use the equivalence provided by Theorem~\ref{thm:trinity} to give a characterization of ab-tree-cut width expressed in terms of an analogue of the cops and robber game, which we call the {\em{cops, dogs, and robber game}}.

The game is played on a graph $G$ by two players, one controlling cops and dogs, and the other controlling the robber. There are two parameters of the game:
\begin{itemize}[nosep]
 \item the number of cops $a$; and
 \item the number of dogs $b$.
\end{itemize}
The game starts with the robber player placing the robber at some vertex $r_0$.
At all times, the cops occupy a set $F$ consisting of at most $a$ edges of $G$. This set is initially empty, that is, $F_0=\emptyset$. Then the players proceed in rounds. Each round $i\in \{1,2,3,\ldots\}$, consists of the following steps: 
\begin{itemize}[nosep]
 \item The cop player announces a set of edges $F_i$ of size $\leq a$ to which the cops will move in this round.
 \item The robber player moves the robber from vertex $r_{i-1}$ to any vertex $r_i$ which is reachable from $r_{i-1}$ by a path in $G$ which does not pass through any edge occupied by a cop that does not move, that is, a path that does not intersect $F_{i-1}\cap F_i$. However, we also require that $r_i$ and $r_{i-1}$ are $3$-edge-connected in $G$.
 \item The cops execute the announced move.
\end{itemize}
Note that so far the dogs do not get to play, but they are important in the winning condition. Namely, after every round the players verify to how many vertices the robber could potentially move, that is, how many vertices of $G$ are $3$-edge-connected with $r_i$ and can be reached from $r_i$ by a path that avoids the edges of $F_i$. If this number is $\leq b$, then the cop player can unleash the $b$ dogs on those vertices and immediately catch the robber, thus winning the game. The robber player wins the game if she can avoid getting caught indefinitely. 

Note that the restriction about $3$-edge-connectedness of the moves essentially means that the robber is confined to the $3$-edge-connected component of $G$ to which $r_0$ belongs.

We say that a graph $G$ is {\em{searchable}} by $a$ cops and $b$ dogs, if there exists a strategy for the cop player to win the game using $a$ cops and $b$ dogs.
The main result of this section is the following equivalence.

\begin{theorem}\label{thm:copsandrobber}
Let $G$ be a graph and $a,b$ be positive integers. Then $G$ is searchable by $<a$ cops and $<b$ dogs if and only if $G$ has a tree-cut decomposition of adhesion-width $<a$ and bag-width~$<b$.
\end{theorem}
\begin{proof}
	First, observe again that if $b=1$, then both statements trivially do not hold for every positive integer $a$. Therefore, we shall assume that $b \geq 2$.
	
	We show the left-to-right implication by proving its contrapositive.
	Assume that $G$ has no tree-cut decomposition with adhesion-width $<a$ and bag-width $<b$. 
	We describe a winning strategy for the robber player to win against $<a$ cops and $<b$ dogs.
	
	By Theorem~\ref{thm:trinity}, in $G$ there exists a bramble $\B$ of adhesion-order $\geq a$ and bag-order $\geq b$. The robber player will maintain the following invariant: if at the end of round $i$ the robber is placed at vertex $r_i$, then there is a slab $(H_i,K_i)\in \B$ such that $r_i\in K_i$ and $(H_i,K_i)$ is not disconnected by $F_i$ (the set of edges occupied by the cops at the end of round $i$). Since $F_0=\emptyset$, to have this invariant satisfied at the start of the game, it suffices that the robber player chooses $r_0$ to be any vertex from the core of any slab $(H_0,K_0)\in\B$. 
	
	We now explain how the invariant is maintained in round $i$ of the game.
	When the cop player announces the new set $F_i$ to which the cops will move, the robber chooses any slab $(H_i,K_i)$ that is not disconnected by $F_i$.
	Such a slab exists by the assumption that the adhesion-order of $\B$ is at least $a$ and $b\geq 2$. Since $\B$ is a bramble, the cores $K_{i-1}$ and $K_i$ intersect. The robber player can therefore choose any vertex $r_i\in K_{i-1}\cap K_i$. As the invariant was satisfied in round $i-1$, the set $F_{i-1}\cap F_i$ (in fact, even $F_{i-1}$) does not disconnect $r_{i-1}$ from $r_i$ within $H_{i-1}$. Since $r_{i-1},r_i\in K_{i-1}$ and $K_{i-1}$ is $3$-edge-connected in $G$, we conclude that it is allowed for the robber to move from $r_{i-1}$ to $r_i$, and this move is duly executed by the robber~player.
	
	To see that in this way the robber player can evade being caught indefinitely, observe that provided the invariant is maintained, after round $i$ the robber is allowed to move from $r_i$ to any vertex of $K_i$. As the bag-order of $\B$ is $\geq b$, we have $|K_i|\geq b$, hence $<b$ dogs are never sufficient to catch the robber if she follows the described strategy.
	
    
    We now proceed to the right-to-left implication. This amounts to describing a strategy for ${<a}$ cops and ${<b}$ dogs to search $G$, assuming that $G$ has a tree-cut decomposition $\T=(T,\X)$ satisfying $\adhw(\T)<a$ and $\bagw(\T)<b$.
    
    Let us root the tree $T$ in an arbitrary node, which naturally imposes an ancestor-descendant relation in $T$. After round $0$, when we have $F_0=\emptyset$, in rounds $1,2,3,\ldots$ the cop player will select nodes $t_1,t_2,t_3,\ldots$ so that each $t_{i+1}$ is a child of $t_{i}$, and play
    $$F_i\coloneqq \adh_{\T}(t_i).$$
    While doing this, the cop player will maintain the following invariant: after round $i$, either she has already won, or the robber must be placed at a vertex that belongs to a bag of a (strict) descendant of $t_i$. 
    
    The strategy for maintaining the invariant is simple. In round $1$ the cop player chooses $t_1$ to be the root of $T$, while in round $i\geq 2$ she chooses $t_i$ to be the child of $t_{i-1}$ such that $r_{i-1}$ belongs to the bag of either $t_{i-1}$ or any of its descendants. We now verify that the invariant is maintained after round $i$; we do this only for $i\geq 2$, as for $i=1$ the check is almost the same. For brevity, let $T_i$ be the subtree of $T$ rooted at $t_i$.
    
    First, observe that either the edge $t_it_{i-1}$ is thin in $\T$, or 
    $$\adh(t_it_{i-1})\subseteq \adh(t_i)\cap \adh(t_{i-1})=F_i\cap F_{i-1}.$$
    In either case, the robber cannot move from $r_{i-1}$ to any vertex $w$ outside of $\bigcup_{s\in V(T_i)} X_s$, because either $w$ and $r_{i-1}$ are not $3$-edge-connected in $G$, or every path connecting $r_{i-1}$ with $w$ intersects $F_i\cap F_{i-1}$. Hence, the robber player must choose $r_i\in \bigcup_{s\in V(T_i)} X_s$. However, if she chooses $r_i\in X_{t_i}$, then she immediately loses after this round: the set of vertices to which the robber can move once the cops are on $F_i$ would be confined to a subset of $X_{t_i}$, which is of size~$<b$. Hence, to avoid being captured the robber player must choose $r_i$ to be a vertex contained in a bag of a strict descendant of $t_i$, and the invariant is maintained.
    
    To see that following the strategy results in catching the robber, observe that eventually the cop player will chose $t_i$ to be a leaf of $T$. Then she wins, as there is no vertex at which the robber can be placed after this round.
\end{proof}

A graph is {\em{$k$-searchable}} if it is searchable by $k$ cops and $k$ dogs.  Theorem~\ref{thm:copsandrobber} then implies the following.

\begin{corollary}
 A graph has ab-tree-cut width $\leq k$ if and only if it is $k$-searchable.
\end{corollary}